\documentclass[12pt]{amsart}

\usepackage{amsmath, amsfonts,amssymb,url,epsfig}
\usepackage[pdftex]{hyperref}
\usepackage{algorithmic}
\usepackage{epstopdf}
\usepackage{color}
\usepackage{bm}
\usepackage{bbm}
\usepackage{enumitem}
\usepackage[margin=1.13in]{geometry}
\graphicspath{{./PIX/}}

\definecolor{blue}{RGB}{0,0,255}
\definecolor{green}{RGB}{50,150,50}
\definecolor{red}{RGB}{255,0,0}
\definecolor{grey}{RGB}{197,77,87}
\definecolor{darkblue}{rgb}{0.0, 0.0, 0.55}
\renewcommand{\emph}[1]{{\color{darkblue}\em #1}}

\newtheorem{theorem}{Theorem}
  \newtheorem{lemma}[theorem]{Lemma}

  \newtheorem{proposition}[theorem]{Proposition}

\newtheorem{algorithm}[theorem]{Algorithm}

  \theoremstyle{definition}
  
  \newtheorem{example}[theorem]{Example}
   \newtheorem{remark}[theorem]{Remark}

\newcommand{\VV}{\mathbb V}
\newcommand{\CC}{\mathbb C}

\newcommand{\QQ}{\mathbb Q}
\newcommand{\ZZ}{\mathbb Z}

\newcommand{\bL}{{\bm L}}

\newcommand{\bff}{\bm f}

\newcommand{\bI}{\bm I}

\newcommand{\ff}{\bm f}

\newcommand{\degXL}[2]{\ensuremath{d(#1,#2)}}

\newcommand{\inrange}[1]{\ensuremath{ = 1, \dots, #1}}

\DeclareMathOperator{\init}{in}
\DeclareMathOperator{\tinit}{t-in}

\DeclareMathOperator{\trop}{trop}

\DeclareMathOperator{\Span}{span}

\begin{document}

\title[]{Beyond Polyhedral Homotopies}
  \author{Anton Leykin}
 \address{School of Mathematics, Georgia Tech,
        Atlanta GA, USA}
\email {[leykin,jyu]@math.gatech.edu}
\thanks{AL is supported by NSF-DMS grant \#1151297}

  \author{Josephine Yu}
\thanks {JY is supported by NSF-DMS grant \#1600569.}
\date{\today}

\begin{abstract}
We present a new algorithmic framework which utilizes tropical geometry and homotopy continuation for solving systems of polynomial equations where some of the polynomials are generic elements in linear subspaces of the polynomial ring. 
 This approach generalizes the polyhedral homotopies by Huber and Sturmfels.  
\end{abstract}

 \maketitle


\section{Introduction}

The \emph{polyhedral homotopy} continuation method of Huber and Sturmfels~\cite{HuberSturmfels}, which is implemented in PHCpack~\cite{V99}, HOM4PS~\cite{HOM4PSwww,chen2014hom4ps}, and PHoM~\cite{gunji2004phom}, is used for computing numerical approximations of all the roots of $n$ polynomial equations in $n$ variables, where it is assumed that each equation has generic coefficients with respect to its monomial support.
We generalize this to
\begin{description}[leftmargin=*]
\item[Setting A] Instead of solving for all solutions in $(\CC^*)^n$, we would like to find solutions lying on a variety $X \subset \CC^n$ defined by the set of polynomials $G$, away from the base locus, where the number of additional given polynomials (other than $G$) is equal to $\dim(X)$ and each of the additional polynomials is generic with respect to its monomial support.
\end{description}
This can be further generalized to
\begin{description}[leftmargin=*]
\item[Setting B] Instead of monomial supports, we can consider arbitrary supports, i.e.\ we would like to find roots on $X$ of polynomials that are generic linear combinations of arbitrary sets
 of polynomials.
\end{description}
We propose a 3-stage framework to solve this problem in the general setting.
\begin{description}[leftmargin=*]
\item[Initialization] Reformulate the problem in Setting B into Setting A as explained in \S\ref{sec:setup}. Then pick---in practice, randomly---a one-parameter system $\bff(t)$ on $\dim(X)$ many polynomials with the specified monomial supports and coefficients that are generic rational powers of the parameter $t$ multiplied by a generic complex number.
\item[Stage 1 (mostly symbolic)] Compute the tropical variety $\trop(X)$.  This usually involves both polyhedral computations and Gr\"obner basis computations.
\item[Stage 2 (polyhedral)] Compute the (transverse) intersection of $\trop(X)$ with  tropical hypersurfaces of polynomials in $\bff(t)$.  
\item[Stage 3 (mostly numerical)] Find the initial terms---in general, distinct truncations---of Puiseux series solutions of the system $G =\bff(t) = 0$ corresponding to the tropical points found in Stage 2.
Track the homotopy paths for $t\in[0,1]$, which are approximated by these truncated Puiseux series in the beginning ($t$ close to $0$) and lead to solutions of a system with generic complex coefficients at the end ($t=1$).
\end{description}
When $X = \CC^n$ in Setting A, the Stage 1 is not needed, and the framework specializes to the polyhedral homotopy approach. 
The precise statements are in the pseudocode of Algorithm~\ref{alg:main}, which  is followed by remarks on currently available tools for implementation.

\section{The problem setup}\label{sec:setup}

The more general setup (Setting B) is as follows.
Let $X = \VV(G) \subseteq \CC^n$.  Let $L_1,L_2,\dots,L_r$ be vector subspaces of $\CC[x]$ spanned by finite sets $F_1,F_2,\dots,F_r$ respectively. Let $\bL := L_1 \times \cdots \times L_r$. Let $Z_{L_i} := \VV(F_i)$ be the \emph{base locus} of the linear spaces $L_i$.  Let $Z_\bL := \bigcup_{i=1}^r Z_{L_i}$.

\smallskip

Our {\bf main goal} is to compute all the points in \mbox{$(X\setminus Z_\bL) \cap \VV(\ff)$} for some generic element $\ff=(f_1,\dots,f_r) \in \bL$.  Our {\bf enumerative goal} is to compute the number $\degXL{X}{\bL}$ of such points, which we assume to be finite.

A discussion of subtleties surrounding the base locus and genericity appears in the Appendix.
We can deal with rational functions in $F_i$'s by clearing the denominators and removing the zero locus of the the denominators from $X\backslash Z_\bL$.

We will now reformulate the Setting B into Setting A.
Let $P = \{h_1,h_2,\dots,h_\ell\}$ be the set of non-monomials in $F_1 \cup \cdots \cup F_r$.  Consider
\begin{equation}\label{eq:replace}
\begin{array}{rcl}
G' &=& G \cup \{z_i - h_i(x)  \mid 1 \leq i \leq \ell \},\\ 
F'_j &=& (\text{$F_j$ with $h_i$ replaced by $z_i$}),\ j\inrange{r},
\end{array}
\end{equation}
which are sets of polynomials in $\CC[x,z]:=\CC[x_1,\dots,x_n, z_1,\dots,z_\ell]$.
Solving the system $f_1=\cdots=f_r=0$ on $X = \VV(G)$ is equivalent to solving $f_1' = \cdots =f_r' = 0$ on the variety $\VV(G')$ where $f_j'$ is obtained from $f_j$ by replacing each $h_i$ with $z_i$.  
The new polynomials $f'_1,\dots,f'_r$ are generic with respect to their monomial support. For the rest of the paper we assume Setting A, that is, $F_i$ forms a {\bf monomial basis} of $L_i$ for each $i\inrange{r}$.  We drop the primes~$'$ for simpler notation.



\begin{example}\label{example:two-circles}
Two generic circles 
\begin{eqnarray*}
\label{equation:two-circles}
a_1(x^2+y^2) + a_2x + a_3 y + a_4 &=& 0\\
a_5(x^2+y^2) + a_6x + a_7 y + a_8 &=& 0
\end{eqnarray*}
intersect in two points in $\CC^2$ although the mixed volume of their Newton polytopes is $4$.  
We rewrite the system as:
\begin{eqnarray*}
\label{equation:two-circles}
z - (x^2+y^2) &=& 0\\
a_1 z + a_2x + a_3 y + a_4 &=& 0\\
a_5 z + a_6x + a_7 y + a_8 &=& 0
\end{eqnarray*}
The original equations are transformed into polynomials that are generic with respect to their monomial supports, but we acquire a new equation whose coefficients may be special.
\qed
\end{example}

To achieve the main goal, we will construct a homotopy with exactly $\degXL{X}{\bL}$ paths to track.  
The number $\degXL{X}{\bL}$ is called the \emph{intersection index} $[L_1,\dots,L_r]$ in~\cite{KavehKhovanskii} and equals the mixed volume of \emph{Newton-Okounkov bodies} associated to $L_1, \dots, L_r$ on $X$.

\section{Algorithmic framework}

Let $\CC\{\!\{t\}\!\}$ be the field of Puiseux series that are convergent on a punctured neighborhood of $0$ in $\CC$.  (See~\cite{Ghys-promenade-book} for a proof of this fact and a historical excursion.)  There is a valuation from $\CC\{\!\{t\}\!\}-\{0\}$ to $\QQ$ given by the leading (lowest) degree.

For an ideal $I$ in $\CC\{\!\{t\}\!\}[x_1,\dots,x_n]$ and a weight vector $\omega \in \QQ^n$, the \emph{$t$-initial ideal} $\tinit_\omega(I)$ is obtained by first taking the usual initial ideal with the min-convention (leading terms are lightest), where the weight of $t$ is $1$ and the weights of $x$'s are given by $\omega$, and then setting $t=1$.  
The $t$-initial ideal is an ideal in $\CC[x]$; it does not involve $t$. For example, $\tinit_{(1,2)}\langle (t+t^2)x+2y+3tx^2+(5t^2+7t^3)\rangle = \langle x+2y+5
\rangle$.

The tropical variety of $I$ is defined as
\[
\trop(I) = \{\omega \in \QQ^n : \tinit_\omega(I) \text{ does not contain a monomial}\}.
\]  
We often write $\trop(X)$ to denote $\trop(I)$ when $X = \VV(I)$, and we write $\trop(f)$ to denote $\trop(\langle f \rangle)$.
The tropical variety is a polyhedral complex, and we can define \emph{multiplicities} on its maximal faces.  See~\cite[Chapter 3]{MaclaganSturmfels} and~\cite{JMM} for details.

\begin{theorem}[Fundamental Theorem of Tropical Algebraic Geometry]\cite[Theorem~3.2.3]{MaclaganSturmfels}
\label{thm:fundThm}
The points in  $\trop(I)$ are exactly the coordinatewise valuations of the Puiseux series points in the variety of $I$.  The multiplicity of each tropical point $\omega \in \trop(I)$ is equal to the number of Puiseux series point with valuation~$\omega$, counted with multiplicities.
\end{theorem}

We will now formulate the the polyhedral homotopy continuation method of Huber and Sturmfels using tropical geometry.  
Given a polynomial system $\ff = (f_1,\dots,f_n) \in (\CC[x_1,\dots,x_n])^n$ with generic coefficients with respect to their monomial supports, we perturb the coefficients by throwing in extra factors of the form $t^\alpha$ where $\alpha$ are arbitrary rational numbers, to obtain a family of systems $\ff(t) = (f_1(t), \dots, f_n(t)) \in (\CC\{\!\{t\}\!\}[x_1,\dots,x_n])^n$.  See (\ref{eq:f(t)}) with $r=n$.
The homotopy continuation approach looks to ``connect'' solutions of the original system $\ff=\ff(1)$ to the Puiseux series solutions of $\ff(t)$ convergent in some neighborhood of $t=0$.

If the exponents of $t$'s are sufficiently generic, then the intersection of tropical hypersurfaces $\trop({f_1(t)}) \cap  \cdots \cap \trop({f_n(t)})$ is finite and transverse, i.e.\ locally at each intersection point it is a transverse intersection of affine spaces.  In this case we have 
\[
\trop(\langle {f_1(t)}, \dots, {f_n(t)} \rangle) = \trop({f_1(t)}) \cap  \cdots \cap \trop({f_n(t)}).
\]

The points of $\trop({f_1(t)}) \cap  \cdots \cap \trop({f_n(t)})$ are most commonly computed by enumerating \emph{mixed cells}  of \emph{the mixed subdivision}, which is the projection of lower convex hull the Newton polyhedron of the product $f_1(t)\cdots f_n(t)$ onto the $x$-coordinates.

The solutions $c$ of the binomial initial system $\tinit_\omega f_1(t) = \cdots = \tinit_\omega f_n(t) = 0$ give us the leading terms $ct^\omega$ of the convergent Puiseux series with valuation $\omega$ satisfying $\ff(t) = 0$. 
Take $c\varepsilon^\omega$ as a numerical approximation of a point satisfying $\ff(\varepsilon)=0$ for a small $\varepsilon>0$. We can numerically track a segment of a real curve $\ff(t)$, $t\in [\varepsilon,1]$ starting at that point and finishing with an approximation of a solution of to the original system $\ff=\ff(1)$. 
This is called the \emph{polyhedral homotopy} because mixed subdivisions of Newton polyhedra play a crucial role.

\smallskip

Getting back to our Setting A, polynomials in the set $G$, e.g.\ $z - (x^2+y^2)$ in Example~\ref{example:two-circles} above,  can have special coefficients, while the others have generic coefficients  with respect to fixed \emph{monomial} supports.  As above, we wish to compute the tropical variety of the system, which should consist of finitely many points.  

For each $i\inrange{r}$, let $f_i = \sum_{x^\alpha\in F_i}a_{i,\alpha} x^{\alpha}$ be complex polynomials with generic coefficients.  Let 
\begin{equation}\label{eq:f(t)}
f_i(t) =  \sum_{x^\alpha\in F_i}a_{i,\alpha}t^{\omega_{i,\alpha}} x^{\alpha}
\end{equation}
for some generic $\omega_{\bullet,\bullet} \in \QQ$.  We refer to the system $\ff(t)= (f_1(t),\dots,f_r(t))$ as a \emph{homotopy}. We recover the original system $\ff \in\bL$ simply by specializing $t=1$. 

We can look at $\ff(t)$ from two different angles:
\begin{enumerate}
\item Consider $\ff(t)$ as a family of systems in $\bL$, parameterized by $t \in \CC\backslash\{0\}$.  
Given a path $\gamma: [0,1]\to\CC\backslash\{0\}$, we get a path $\ff \circ \gamma : [0,1] \rightarrow \bL$.  Suppose the path in $\bL$ does not go through the branch locus of the projection $\pi: I_{X,\bL} \to \bL$ where 
\[I_{X, \bL} = \{(f_1,\dots,f_r, z) \mid z \in X \setminus Z_\bL \text{ and all } f_i \text{ vanish at }z\} \subset \bL \times X \]
is the incidence variety.  Then the homotopy $\ff(t)$ induces smooth \emph{homotopy paths} $\pi^{-1}(\ff(\gamma(\tau)))$, $\tau\in[0,1]$, that give a one-to-one correspondence between starting roots $\pi^{-1}(\ff(t_0))$ and target roots $\pi^{-1}(\ff(t_1))$, where $t_0=\gamma(0)$ and $t_1=\gamma(1)$.

Once we find an appropriate path $\gamma$ from some $t_0$, where $\pi^{-1}(\ff(t_0))$ is known, to $t_1=1$, we achieve our main goal by \emph{homotopy  continuation} along $\gamma$.

\item Consider $\ff(t)$ as a polynomial system over the \emph{Puiseux series} $\CC\{\!\{t\}\!\}$ with $\degXL{X}{\bL}$ many roots over $\CC\{\!\{t\}\!\}$.  
\end{enumerate}

The second point of view relates to the first as follows. The Puiseux series roots in (2) \emph{converge} in some punctured neighborhood of $t=0$ in $\CC$. Thus, if one can approximate these roots close to $t=0$, one can find a starting point $t_0$ for (1) along with approximations of the starting roots.  This requires approximating the Puiseux series roots, which is explained in Remark~\ref{re:truncations}.

\begin{remark}\label{re:really-lucky}
One can show that the (really) ``unlucky'' vectors of coefficients $a_{\bullet,\bullet}$ in the construction of $\ff(t)$ --- when $\ff(t)$ for some $t\in(0,1]$ intersects the ramification locus of $\pi$ --- is contained in a Zariski closed set of \emph{real} codimension 1 in the ambient (real) coefficient space (with each $a_{\bullet,\bullet}\in\CC$ contributing  two coordinates $\operatorname{Re}(a_{\bullet,\bullet})$ and $\operatorname{Im}(a_{\bullet,\bullet})$).
This implies that with generic choices of $a_{\bullet,\bullet}$, the \emph{real line segment path} running over $t\in[\varepsilon, 1]$ for a small $\varepsilon>0$ is ``lucky''. 
\qed
\end{remark}

\begin{lemma}
\label{lem:intersection}
With the notation as above,
if the coefficients $f_1(t),\dots,f_r(t)$ have sufficiently generic valuations, then 
\[
\trop \langle G \cup \{f_1(t),\dots,f_r(t) \} \rangle = \trop \langle G \rangle \cap \trop \langle f_1(t)\rangle \cap \cdots \cap \trop \langle f_r(t) \rangle.
\]
\end{lemma}

\begin{proof}
The inclusion  $\trop \langle G \cup \{f_1(t),\dots,f_r(t) \} \rangle \subseteq \trop \langle G \rangle \cap \trop \langle f_1(t)\rangle \cap \cdots \cap \trop \langle f_r(t) \rangle$ is always true by the definition of tropical varieties,  but the containment may be strict.  However, when the coefficients of $f_1(t),\dots,f_n(t)$ have sufficiently generic valuations, the intersections are transverse, so the result follows from~\cite[Lemma 3.2]{BJSST}. 
\end{proof}

The multiplicities of the transverse 
intersection points can be computed using integer linear algebra~\cite[formula~(3)]{stableIntersection}.

\begin{figure}[h]
\begin{tabular}{cc}
\begin{minipage}{0.4\textwidth}
\begin{center}
\includegraphics[scale=0.4]{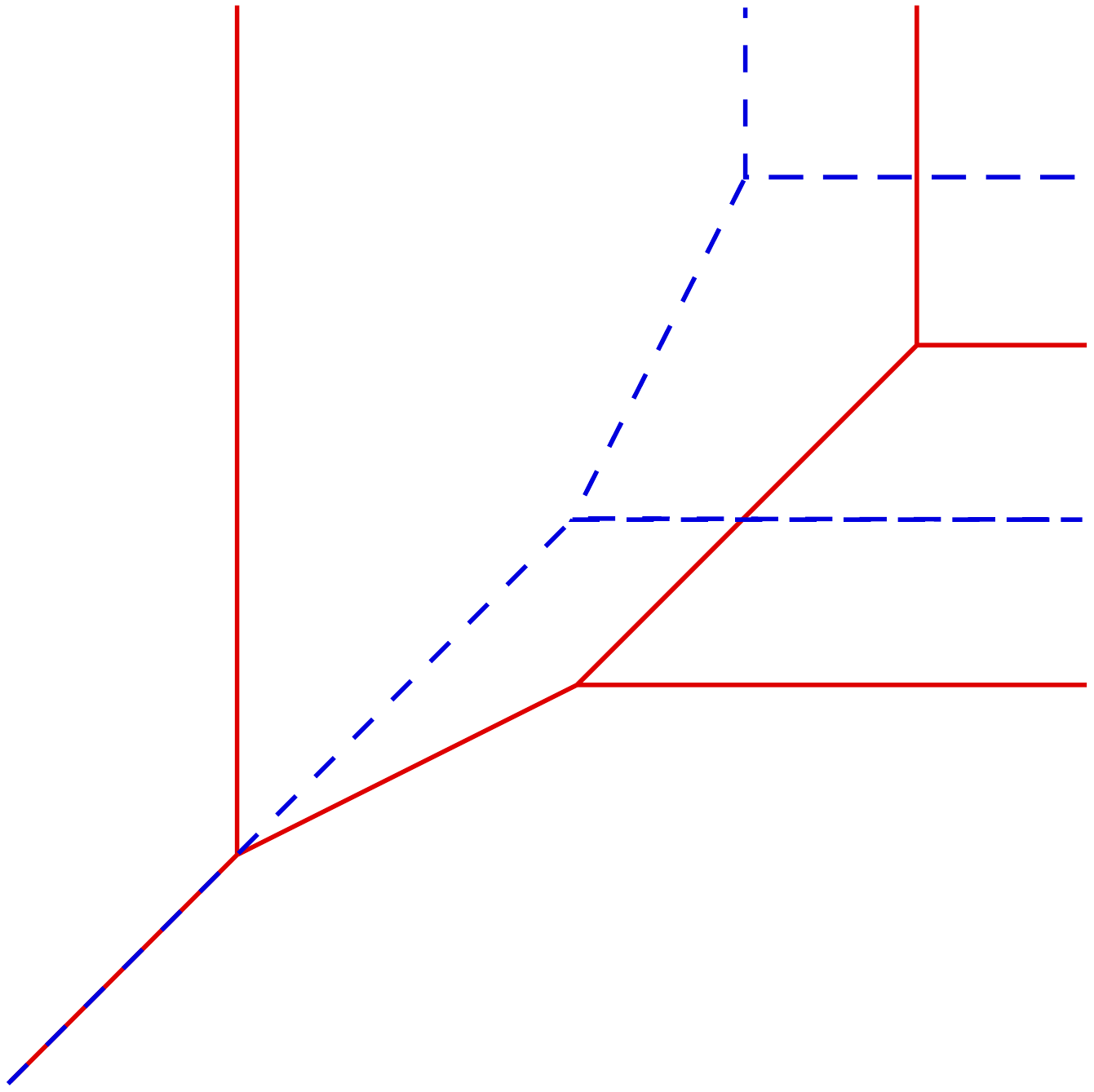}
\vspace{-2.8in}
\end{center}
\end{minipage} &
\begin{minipage}{0.4\textwidth}
\begin{center}
\includegraphics[width=6cm]{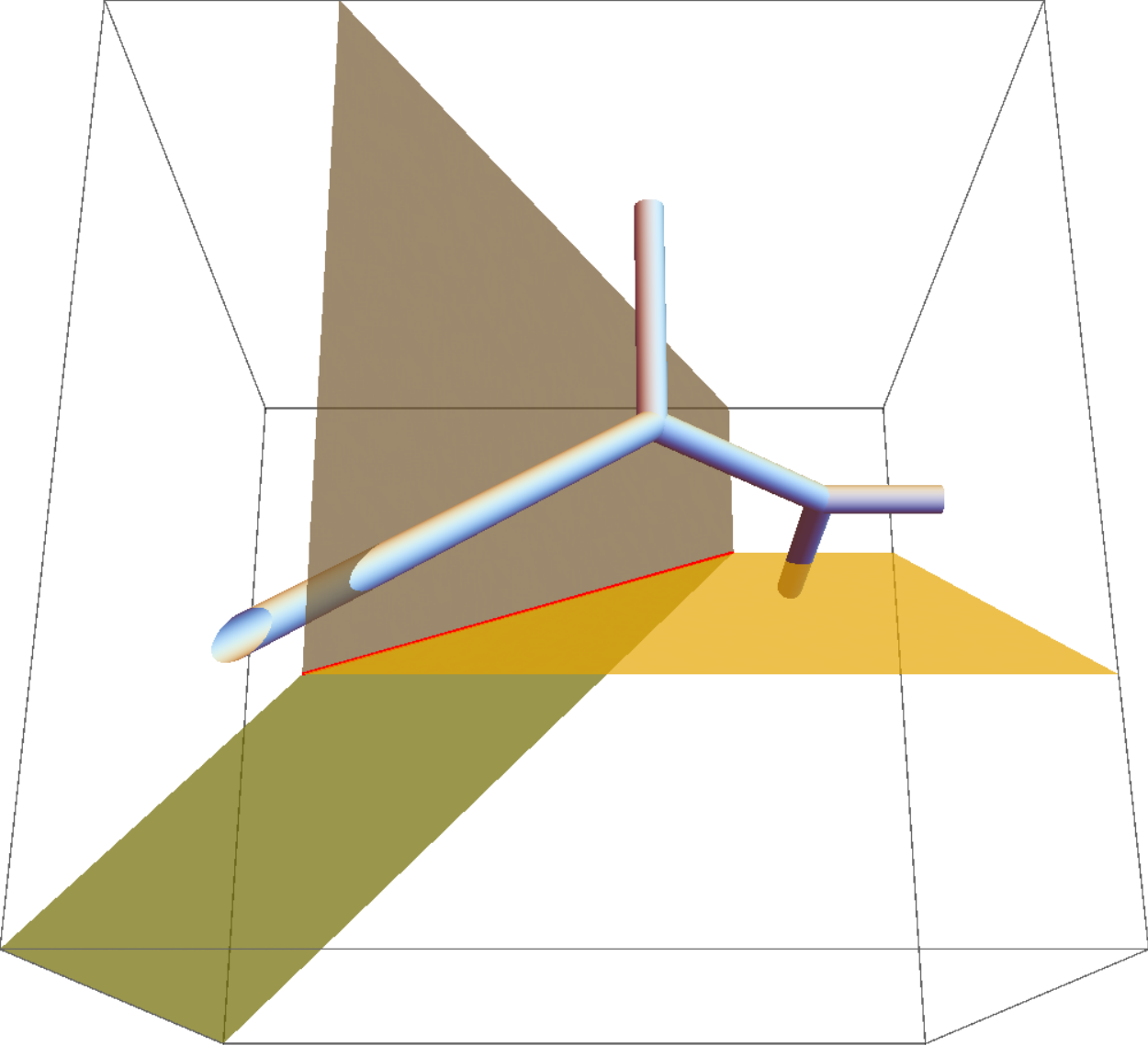}
\end{center}
\end{minipage}
\end{tabular}
\caption{Left: Tropicalizations of the equations for two circles in Example~\ref{example:two-circles} always intersect non-transversely, even when the coefficients $a_1,\dots,a_8$ are generic Puiseux series. 
\newline
Right: The hypersurface $\trop(x^2+y^2-z)$ and a general tropical line, defined by two generic linear equations, intersect transversely at two points.} 
\label{figure:trop-two-circles} 
\end{figure}

For $F_i$'s that are not monomial, that is, for polynomials that are generic with respect to non-monomial supports, the intersection of tropicalizations is \emph{not} necessarily transverse.   They become transverse to $\trop(G')$ after the reformulation (\ref{eq:replace}).  See Figure~\ref{figure:trop-two-circles}.



The valuations of the Puiseux series roots are provided by tropical computations (see Theorem~\ref{thm:fundThm}), while the leading coefficients at a tropical point $\omega \in \trop \langle G \rangle \cap \trop \langle f_1(t)\rangle \cap \dots \cap \trop \langle f_r(t) \rangle$ are given by the root(s) of the \emph{$t$-initial ideal}
$$J_\omega=\tinit_{\omega}\langle G \cup \{f_1(t),\dots,f_r(t)\} \rangle.$$  

The next statement follows from \cite[Lemma~3.1]{stableIntersection}.  
\begin{lemma}\label{lemma:stable}  Let $J_\omega=\tinit_{\omega}\langle G \cup \{f_1(t),\dots,f_r(t)\} \rangle$, then
$$J_\omega = \init_\omega \langle G \rangle + \tinit_\omega \langle f_1(t) \rangle + \cdots + \tinit_\omega \langle f_r(t) \rangle.$$
\end{lemma}

The roots $c\in\VV(J_\omega)$ give us the leading coefficients of Puiseux series roots. The lemma shows that the $J_\omega$ may be generated by simpler polynomials than the original polynomials in $G \cup \{f_1(t),\dots,f_r(t)\}$.
In the original polyhedral homotopies  the ideal $J_\omega$ is binomial and all roots are regular and easy to obtain.

\begin{remark}\label{re:truncations}
In general, distinct Puiseux series roots may share the same leading terms. This happens exactly when a root $c\in \VV(J_\omega)$ is multiple. 

For a multiple root $c$, one needs to find more terms in the \emph{truncated Puiseux series} $$s(t)=ct^\omega + \text{(higher order terms)}, $$ 
so that these are distinct for distinct Puiseux series roots. 
Note that all $s(t)$ and $f_i(t)$'s are polynomials in $\CC[x,t^{1/N}]$ for some positive integer $N$.

The most comprehensive algorithmic treatment of this is can be found in~\cite{JMM}.    
\qed
\end{remark}
\medskip

\begin{algorithm}[Main algorithm]~\label{alg:main}
\begin{algorithmic}[1]
\renewcommand{\algorithmicrequire}{\textbf{Input:}}
\renewcommand{\algorithmicensure}{\textbf{Output:}}
\REQUIRE $G$, a collection of polynomials;\\
$F_i$, $i\inrange{r}$, sets of monomials.  
\ENSURE 
generic $f_i\in L_i = \Span F_i$ for $i\inrange{r}$;\\
approximations $S$ to the points of $\VV(G,f_1,\ldots,f_r)$. 
\smallskip
\hrule
\smallskip
\STATE Compute $\trop(X) = \trop \langle G \rangle$. \label{line:tropX}
\STATE Construct $\ff(t)$, of the form $f_i(t) := \sum_{x^\alpha\in F_i}a_{i,\alpha}t^{\omega_{i,\alpha}} x^{\alpha}$, $i\inrange{r}$.    
\STATE Compute $W = \trop(X) \cap \trop \langle f_1(t)\rangle \cap \cdots \cap \trop \langle f_r(t) \rangle$. \label{line:transverse-intersection}
\STATE $S := \emptyset$
\FOR{every point $\omega \in W$}
  \STATE Construct truncations  $S_\omega$ of Puiseux series roots of $G=\ff(t)=0$. \label{line:truncations}
  \FOR{$s(t) \in S_\omega$}
    \STATE Pick $\varepsilon = \varepsilon(\omega,s(t)) > 0$ close to $0$, let $\tilde x_\varepsilon = s(\varepsilon)$. \label{line:epsilon}
    \STATE  $S :=S \cup \{\tilde x_1\}$,  where $\tilde x_1$ is the output of a homotopy continuation algorithm tracking roots of $G=\ff(t)=0$, $t\in[\varepsilon,1]$, starting with $\tilde x_\varepsilon$.
  \ENDFOR
  \RETURN $f_i := f_i(1)$, $i\inrange{r}$, and $S$.   
\ENDFOR
\smallskip
\hrule
\smallskip
\end{algorithmic}
\end{algorithm}

The following points remained unsaid in the pseudocode:
\begin{itemize}
\item The software Gfan~\cite{gfan} can compute $\trop(X)$ in Line~\ref{line:tropX} when $G$ has rational coefficients.
\item The computation of $\trop(X)$ may involve Gr\"obner bases, while the transverse intersection in Line~\ref{line:transverse-intersection} does not.
\item Jensen's \emph{tropical homotopy continuation}~\cite{jensen2016tropical} may be useful for Line~\ref{line:transverse-intersection}.
\item  Solutions to $J_\omega = \init_\omega \langle G \rangle + \tinit_\omega \langle f_1(t) \rangle + \cdots + \tinit_\omega \langle f_r(t) \rangle$ provide the leading coefficients of $S_\omega$. They give distinct $s(t)$ in Line~\ref{line:truncations}, unless some solutions are multiple. See Remark~\ref{re:truncations}. 
\item The ideal $\init_\omega \langle G\rangle$ is a byproduct of the computation of $\trop(X)$ in Line~\ref{line:tropX}. 
\item A practical way to pin down $\varepsilon = \varepsilon(\omega,s(t))$ in Line~\ref{line:epsilon} is out of the scope of this article; we envision obtaining $\varepsilon$ with heuristics that depend on $\omega$ and $s(t)$. Such $\varepsilon$ exists according to Remark~\ref{re:really-lucky} and the discussion preceding it.
\item We also purposefully omit the discussion of how one tracks a homotopy path in practice. The machinery of \emph{numerical homotopy continuation} is well established with several books (e.g., \cite{Morgan87,AG03} and several more modern ones) devoted to its details.  
\end{itemize}

Algorithm~\ref{alg:main} achieves our {\bf main goal}. The {\bf enumerative goal} is achieved by executing it until Line~\ref{line:transverse-intersection} and then computing the degrees of $J_\omega$. The sum of these degrees is $\degXL{X}{\bL}$. 


\section{Conclusion}

In Setting B we construct an \emph{optimal} homotopy, optimal in the sense that the number of homotopy paths is equal to the number of solutions generically.
Our method combines symbolic, polyhedral, and numerical parts. 
One potential strength is that, for a concrete polynomial system, one can distribute the load between these parts to avoid bottlenecks or to decrease the generic solution count.    

Indeed, polynomials in any given set of equations can be divided into two groups, $G$ and $\ff$, fixing the variety $X=\VV(G)$, on which the roots of $\ff$ are sought. For each polynomial $f_i$ one can decide on which ingredients $F_i$ this polynomial is ``made of''. The list of ingredients may be either inherent to the given problem or be a subject of experimentation. This gives a lot of flexibility.

\appendix
\setcounter{secnumdepth}{0}
\section{Appendix: Base locus}

When we say a statement is true ``for \emph{generic}  $y$ in $Y$'' or ``for \emph{general} $y$ in $Y$'', we mean that it is true for all elements $y$ in some Zariski open dense subset of $Y$, not to be confused with the ``generic point'' of a scheme.  We say that a polynomial $f$ is \emph{generic} or \emph{general} with respect to support $s_1, \dots, s_k$ if $f = c_1 s_1 + \cdots + c_k s_k$ where the coefficients $(c_1,\dots,c_k)$ avoid a Zariski closed proper subset of $\CC^k$ which depends on the context. One should assume that an explicit description of this exceptional set is hard to acquire algorithmically.

In Bernstein--Khovanskii--Kushnirenko theorems and in the original polyhedral homotopies, the goal is to compute (the number of) solutions in the algebraic torus $(\CC^*)^n$, of polynomials that are generic with respect to their monomial supports.  In other words, we choose a monomial basis $\{m_1, m_2,\dots,m_{k_i}\}$ for each linear system $L_i$ and remove the \emph{union} of their hypersurfaces from $X = \CC^n$, obtaining $\CC^n \setminus \bigcup_{j=1}^{k_i} \VV(m_j) = (\CC^*)^n$ if all variables appear in the monomial basis.  
Here we remove the base locus instead, which is the \emph{intersection} of the hypersurfaces.  The following argument shows that these two settings are equivalent.
For generic polynomials the solution set does not depend on the choice of set removed as long as it has smaller dimension and contains the base locus.

\begin{proposition}
\label{prop:baseLocus}
Suppose that  $\VV(f_1, \dots, f_r) \cap X$ is finite for generic  $(f_1,\dots,f_r) \in \bL$. Let $Z \subset \CC^n$ be a variety such that $Z_\bL \subset Z \subset X$ and $\dim Z < \dim X = r$.   Then
\begin{equation}
\label{eqn:boundary}
\VV(f_1,\dots,f_r) \cap (X \setminus Z_\bL) = \VV(f_1, \dots, f_r) \cap (X \setminus Z)
\end{equation}
for generic $(f_1,\dots,f_r) \in \bL$.
\end{proposition}


\begin{proof}
The idea is that for generic choices of $(f_1,\dots,f_r) \in \bL$, each $f_i$ cuts down the dimension of $Z \setminus Z_\bL$ by one, and we assumed that $\dim Z < r$, so we should have
\[
\VV(f_1,\dots,f_r) \cap (Z \setminus Z_\bL) = \varnothing,
\]
which implies~\eqref{eqn:boundary}. We will make this precise.

 If $Z = Z_\bL$, then there is nothing to prove.  Suppose $Z \supsetneq Z_\bL$.
Let $U$ be a Zariski open dense subset of $\bL$ such that $\VV(f_1, \dots, f_r) \cap X$ is finite for each  $(f_1,\dots,f_r) \in U$.
Let \[
  W = \{(f_1,\dots,f_r) \in U \mid \VV(f_1,\dots,f_i) \cap (Z \setminus Z_\bL) =\emptyset\}.
\]
We wish to show that $W$ contains a Zariski open dense subset of $\bL$.

Note that $W$ is constructible, since for the incidence variety
\[
\bI_Z = \{(f_1,\dots,f_r, z) \mid z \in Z \setminus Z_\bL \text{ and } f_1(z) = \cdots = f_r(z) = 0  \} \subset \bL \times (Z \setminus Z_\bL),
\]
with  the projection $\pi$ onto $\bL$ we have $W = U \setminus \pi(\bI_Z)$.  

For a point $z \in Z \setminus Z_{L_1}$, the set of $f_1 \in L_1$ satisfying $z \notin \VV(f_1)$ is the complement of a hyperplane in $L_1$.  Since the degrees of the $f_1$'s in $L_1$ are bounded, the condition that $\VV(f_1)$ does not contain any component of $Z \setminus Z_\bL$ can be translated as $f_1$ not vanishing on a finite set of points; hence the set of such $f_1$'s is a Zariski open dense subset of $L_1$.  By Krull's Hauptidealsatz, if $\VV(f_1)$ does not contain any component of $Z \setminus Z_\bL$, then $\VV(f_1) \cap (Z \setminus Z_\bL)$ is either empty or has dimension $\leq \dim(Z) -1$.  

For a fixed $f_1$, by a similar argument, there is a Zariski open dense subset of $f_2 \in L_2$ such that $\VV(f_1, f_2) \cap (Z \setminus Z_\bL)$ has dimension $\leq \dim(Z) -2$, and so on.  This shows that the set $W$ is dense in $\bL$. Since it is constructible, it contains a Zariski  open dense subset.
\end{proof}

\bibliographystyle{amsalpha}
\bibliography{mybib}

\newcommand{\etalchar}[1]{$^{#1}$}
\def\cprime{$'$}
\providecommand{\bysame}{\leavevmode\hbox to3em{\hrulefill}\thinspace}
\providecommand{\MR}{\relax\ifhmode\unskip\space\fi MR }
\providecommand{\MRhref}[2]{%
  \href{http://www.ams.org/mathscinet-getitem?mr=#1}{#2}
}
\providecommand{\href}[2]{#2}
\begin{thebibliography}{GKK{\etalchar{+}}04}

\bibitem[AG03]{AG03}
E.L. Allgower and K.~Georg, \emph{Introduction to numerical continuation
  methods}, Classics in Applied Mathematics, vol.~45, SIAM, 2003.

\bibitem[BJS{\etalchar{+}}07]{BJSST}
T.~Bogart, A.~N. Jensen, D.~Speyer, B.~Sturmfels, and R.~R. Thomas,
  \emph{Computing tropical varieties}, J. Symbolic Comput. \textbf{42} (2007),
  no.~1-2, 54--73. \MR{2284285 (2007j:14103)}

\bibitem[CLL14]{chen2014hom4ps}
Tianran Chen, Tsung-Lin Lee, and Tien-Yien Li, \emph{{Hom4PS-3}: a parallel
  numerical solver for systems of polynomial equations based on polyhedral
  homotopy continuation methods}, International Congress on Mathematical
  Software, Springer, 2014, pp.~183--190.

\bibitem[Ghy]{Ghys-promenade-book}
{\'E}tienne Ghys, \emph{A singular mathematical promenade}, arXiv preprint
  arXiv:1612.06373.

\bibitem[GKK{\etalchar{+}}04]{gunji2004phom}
Takayuki Gunji, Sunyoung Kim, Masakazu Kojima, Akiko Takeda, Katsuki Fujisawa,
  and Tomohiko Mizutani, \emph{{PHoM}: a polyhedral homotopy continuation
  method for polynomial systems}, Computing \textbf{73} (2004), no.~1, 57--77.

\bibitem[HS95]{HuberSturmfels}
Birkett Huber and Bernd Sturmfels, \emph{A polyhedral method for solving sparse
  polynomial systems}, Math. Comp. \textbf{64} (1995), no.~212, 1541--1555.
  \MR{1297471 (95m:65100)}

\bibitem[Jen]{gfan}
Anders~N. Jensen, \emph{{G}fan, a software system for {G}r{\"o}bner fans and
  tropical varieties}, Available at
  \url{http://home.imf.au.dk/jensen/software/gfan/gfan.html}.

\bibitem[Jen16]{jensen2016tropical}
Anders~Nedergaard Jensen, \emph{Tropical homotopy continuation}, 2016.

\bibitem[JMM08]{JMM}
Anders~Nedergaard Jensen, Hannah Markwig, and Thomas Markwig, \emph{An
  algorithm for lifting points in a tropical variety}, Collect. Math.
  \textbf{59} (2008), no.~2, 129--165. \MR{2414142 (2009a:14077)}

\bibitem[JY16]{stableIntersection}
Anders Jensen and Josephine Yu, \emph{Stable intersections of tropical
  varieties}, J. Algebraic Combin. \textbf{43} (2016), no.~1, 101--128.
  \MR{3439302}

\bibitem[KK12]{KavehKhovanskii}
Kiumars Kaveh and A.~G. Khovanskii, \emph{Newton-{O}kounkov bodies, semigroups
  of integral points, graded algebras and intersection theory}, Ann. of Math.
  (2) \textbf{176} (2012), no.~2, 925--978. \MR{2950767}

\bibitem[LLT]{HOM4PSwww}
T.~L. Lee, T.~Y. Li, and C.~H. Tsai, \emph{{HOM4PS-2.0}: a software package for
  solving polynomial systems by the polyhedral homotopy continuation method},
  Available at http://hom4ps.math.msu.edu/HOM4PS\_soft.htm.

\bibitem[Mor87]{Morgan87}
Alexander Morgan, \emph{Solving polynomial systems using continuation for
  engineering and scientific problems}, Prentice Hall Inc., Englewood Cliffs,
  NJ, 1987. \MR{MR1049872 (91c:00014)}

\bibitem[MS15]{MaclaganSturmfels}
Diane Maclagan and Bernd Sturmfels, \emph{Introduction to {T}ropical
  {G}eometry}, Graduate Studies in Mathematics, vol. 161, American Mathematical
  Society, Providence, RI, 2015.

\bibitem[Ver99]{V99}
J.~Verschelde, \emph{Algorithm 795: {PHC}pack: A general-purpose solver for
  polynomial systems by homotopy continuation}, ACM Trans. Math. Softw.
  \textbf{25} (1999), no.~2, 251--276, Available at
  {http://www.math.uic.edu/$\sim$jan}.

\end{thebibliography}

\end{document}